\documentclass[11pt]{amsart}
\usepackage{amsmath}
\usepackage{amsfonts}
\usepackage{amssymb}
\newtheorem{thm}{Theorem}[section]

\newtheorem{cor}[thm]{Corollary}

\newtheorem{remark}[thm]{Remark}

\newtheorem{prop}[thm]{Proposition}
\newtheorem{defn}[thm]{Definition}
\newtheorem{ex}[thm]{Example}
\newtheorem{prob}[thm]{Problem}
\newcommand{\bb}[1]{\mathbb{#1}}
\newcommand{\cl}[1]{\mathcal{#1}}

\begin{document}
\title[Injectivity]
{Injectivity and Projectivity in Analysis and Topology}

\author[Don Hadwin]{Don Hadwin}
\address{Don Hadwin: Department of Mathematics, University of New
  Hampshire, Durham, NH 03824}
\email{don@math.unh.edu}
\author[Vern I. Paulsen]{Vern I. Paulsen}
\address{Vern I. Paulsen: Department of Mathematics, University 
of Houston,
4800 Calhoun Road, Houston, TX 77204-3008 U.S.A.}
\email{vern@math.uh.edu, URL: http://www.math.uh.edu/$\sim$vern/}
\thanks{Research supported in part by a grant from the National 
Science
Foundation}
\date{\today}

\begin{abstract}We give new proofs of many injectivity results in
  analysis that make more careful use of the duality between abelian
  C*-algebras and topological spaces.  We then extend many of these
  ideas to incorporate the case of a group action. This approach gives
  new insight into Hamana's theory of G-injective operator
spaces and G-injective envelopes. Our new proofs of these classic
results, use only topological methods and eliminate the need for
results from the theory of Boolean algebras and AW*-algebras.
\end{abstract}
\maketitle
\section{Introduction}
This paper has several goals. The first is primarily pedagogical and expository.
There are many results concerned with determining the injective objects in various settings in analysis, which we unify into one fairly straightforward result, whose proof exploits more fully the duality between topological spaces and abelian C*-algebras than earlier proofs. In addition, we show that a number of other results in the theory of operator algebras can be given more simple proofs by exploiting this duality.

Recall that there is a contravariant functor between the category whose objects are compact, Hausdorff spaces with morphisms the continuous maps between them and the category whose objects are unital abelian C*-algebras with morphisms the *-homomorphisms. Because this functor reverses arrows, it carries the diagram that defines injective objects to the diagram that defines projective objects.
This correspondence was utilized by A. Gleason\cite{Gl}, who was the
first to define projective topological spaces and then used this
functor to make some observations about injectivity for abelian
C*-algebras. This work was built upon by
Gonshor\cite{Go1},\cite{Go2}. However, their work used Stone's \cite{St} characterization of the maximal ideal space of complete Boolean algebras and Birkhof's characterization of the Boolean algebra generated by the regular sets in a topological space, which are both fairly non-trivial results. In particular, these results introduce non-topological methods into their work.
Our first goal is to re-derive their work in a purely elementary, topological manner.

Once this is done, several fairly deep results, such as the Nachbin-Goodner-Kelley \cite{Na}, \cite{Goo}, \cite{Ke}, characterization of the injective objects in the category of Banach spaces and contractive linear maps and the characterization of the maximal ideal spaces of abelian AW*-algebras can be combined into one relatively short theorem, whose proof uses more fully the correspondences between injectivity and projectivity. By carrying out all of our constructions in the toplogical setting, we also clarify the construction and properties of the injective envelopes, by dually, first constructing a minimal projective cover.

Finally, using projectivity of topological spaces, we can give simpler
proofs of some of the results of \cite{DP1}, \cite{DP2} and \cite{Pea} concerning continuous matrix-valued functions defined on Stonian spaces. These latter proofs borrow ideas from Azoff's work \cite{Az}, which originally simplified some results in measurable selection. In essence, Azoff's work recognized the utility of projectivity arguments in measurable selection theory.

Our second motivation is to better understand Hamana's work on
G-injective operator systems and G-injective envelopes \cite{H3} \cite{H4}, by first working out the corresponding parallel theory of G-projective topological spaces, where G denotes a discrete group acting
on all of the spaces. Here we are less successful.

Even in the case of abelian C*-algebras there are subtle differences
and similarities between injectivity and G-injectivity that we will
explore.

In particular, if $X$ is a compact, Hausdorff space, then injectivity of the C*-algebra of continuous functions on $X, C(X)$ in several different categories
is known to be equivalent to the space, $X$, being {\em projective} in the sense of A. Gleason\cite{Gl}. However, we shall show that these analogies fail in the presence of a group action. If $X$ is G-projective, in a sense analogous to Gleason's, then $C(X)$ is G-injective in Hamana's sense, but not conversely. Thus, G-projectivity of $X$ is a stronger condition than G-injectivity of $C(X)$, while when there is no group action these two conditions are equivalent.

Thus, in an attempt to understand these distinctions more clearly, we revisited the circle of results surrounding injectivity of $C(X)$ and projectivity of $X.$ While section 2 is primarily expository, we feel that our proofs of the main results in the area are conceptually simpler and serve to better clarify why several definitions of injectivity in various categories all coincide with projectivity of the topological space. Moreover, this presentation serves to motivate the results of the later sections in the presence of group actions.

In section 3, we introduce the concept of a G-projective topological space and attempt to prove as many parallels with the results of section 2 as possible. 

In section 4, we apply the results of section 3 to the study of G-injective operator systems and G-injective envelopes.


\section{Projective Spaces}

In this section, we give an alternative presentation of Gleason's theory of projective topological
spaces\cite{Gl} that avoids the use of Boolean algebras and uses
entirely topological means. Similarly, we develop
the concept of an essential, projective cover of a topological space. We then use duality to apply these results to abelian
C*-algebras and show that the essential projective cover of a space $X$,
corresponds to the injective envelope of $C(X)$. Somewhat similar
ideas are worked out by Gonshor\cite{Go1}, but he immediately left the
category of topological spaces and worked out the injective envelope
in the category of abelian C*-algebras first and then used duality to return to topological spaces.

We begin by studying projectivity in the category whose objects are compact, Hausdorff spaces and whose
maps are continuous functions.

\begin{defn}[Gleason]
A compact, Hausdorff topological space $P$ is {\bf projective}, if for any pair of topological spaces,
$X, Y$ and pair of continuous maps $h: Y \to X$ and $f: P \to X$, with $h$ onto, there exists a continuous
map $r: P \to Y$ such that $h \circ r(p) = f(p)$ for every $p \in P.$ We will call r a {\bf lifting} of f.
\end{defn}

\begin{defn} A compact, Hausdorff space $X$ is {\bf extremally
    disconnected} or {\bf Stonian} if the closure of every open set is open.
\end{defn}

Stone proved that the maximal ideal space of a complete Boolean algebra is extremally disconnected, which is why these spaces are also called Stonian.

The following result of Gleason's uses only elementary topological
methods.

\begin{thm}[Gleason] A compact, Hausdorff space is projective if and
  only if it is extremally disconnected.
\end{thm}

The following results summarize the importance of these concepts and organizes several results that appear
in different places. The direct proofs that we supply below have the
advantage of not using any of the theory of Boolean algebras, but in turn can be used to prove Stone's theorem(see Remark 2.5).

\begin{thm}
Suppose $\mathcal{A}$ is a unital commutative C*-algebra with maximal ideal
space $X.$ The following are equivalent:

\begin{enumerate}
\item $\mathcal{A}$ is an injective operator system,

\item $\mathcal{A}$ is injective in the category of commutative unital
C*-algebras and $\ast$-homomorphisms,

\item $X$ is projective,

\item Whenever $\mathcal{E},\mathcal{F}$ are nonempty subsets of
$\mathcal{A}^{sa}$ such that, $\mathcal{E}$ $\leq\mathcal{F},$ there is an
$a\in\mathcal{A}^{sa}$ such that $\mathcal{E}\leq a\leq\mathcal{F},$

\item $\mathcal{A}$ is injective in the category of Banach spaces and contractive, linear maps.
\bigskip
\end{enumerate}
\end{thm}

\begin{proof}

{\bf $\left(  1\right)  \Rightarrow\left(  4\right)  .$} Suppose
$\mathcal{A=C}\left(  X\right)  $ is an injective operator system. Since
$C\left(  X\right)  \subset\ell^{\infty}\left(  X\right)  =\left\{
f|f:X\rightarrow\mathbb{C},f\text{ bounded}\right\}  ,$ there is a completely
positive linear map $\varphi:\ell^{\infty}\left(  X\right)  \rightarrow
C\left(  X\right)  $ that fixes each element of $C\left(  X\right)  .$ Suppose
$\mathcal{E}$ and $\mathcal{F}$ are as in $\left(  4\right)  .$ We can choose
$g\in\ell^{\infty}\left(  X\right)  $ such that, for every $e\in\mathcal{E}$
and every $f\in\mathcal{F},$ $e\leq g\leq f,$ which implies $e=\varphi\left(
e\right)  \leq\varphi\left(  g\right)  \leq\varphi\left(  f\right)  =f.$

\item[$\left(  4\right)  \Rightarrow\left(  3\right)  $] Suppose $U$ is an
open subset of $X$ and let
\[
\mathcal{E=}\left\{  f\in C\left(  X\right)  ^{sa}:f\leq\chi_{\bar{U}%
}\right\}  ,
\]
and let%
\[
\mathcal{F=}\left\{  g\in C\left(  X\right)  :\chi_{\bar{U}}\leq g\right\}  .
\]
It follows from $\left(  4\right)  $ that there is an $h\in C\left(  X\right)
$ such that
\[
f\leq h\leq g
\]
for every $f\in\mathcal{E}$ and every $g\in\mathcal{F}.$ It is clear that
$g=\chi_{\bar{U}}.$ Hence $\bar{U}$ is open. Therefore, $X$ is
extremally disconnected and hence projective.

{\bf $\left(  3\right)  \Rightarrow\left(  2\right)  .$} Suppose $Y$ is
compact Hausdorff space, $\mathcal{B}$ is a unital C*-subalgebra of $C\left(
Y\right)  ,$ and $\pi:\mathcal{B}\rightarrow C\left(  X\right)  $ is a
$\ast$-homomorphism. Identifying $\cl B = C(Z)$ for some compact,
Hausdorff space $Z,$ the inclusion of $C(Z)$ into $C(Y)$ is given by
composition with
a continuous function, $f: Y \to Z$, which is onto because the
inclusion is one-to-one. Similarly, there exists $h: X \to Z$ such
that $\pi$ is given by composition with $h.$ Since $X$ is projective,
there exists a lifting of $h, r: X \to Y.$ If we define $\rho: C(Y)
\to C(X)$ to be composition with $r,$ then this homomorphism extends
$\pi.$
Hence (2) follows.

{\bf $\left( 2\right) \Rightarrow\left( 3\right).$} Suppose that we are given compact, Hausdorff
spaces $Y,Z$ a continuous onto function, $f:Z \to Y$ and a continuous function $h: X \to Y.$
By composition, $f$ induces a one-to-one *-homomorphism $\pi_f: C(Y) \to C(Z)$ and $h$ induces
a *-homomorphism $\pi_h : C(Y) \to C(X)$. By the injectivity of $C(X)$ in the category of abelian
$C^*$-algebras, there exists a *-homomorphism $\pi: C(Z) \to C(X)$ extending $\pi_f$. By Poincare
duality, this *-homomorphism is induced by a continuous function $s: X \to Z,$ which can be seen to be
a lifting of $h.$ Hence, $X$ is projective.

{\bf $\left(  3\right)  \Rightarrow\left(  5\right).$} Let $B \subset C$ be Banach spaces and let 
$T: B \to C(X)$ be a contractive linear map. Let $B^*_1$ denote the unit ball of the dual of $B$ equipped with 
the wk*-topology. Composing $T$ with the evaluation functionals on $X$ yields a continuous map, $h:X \to B^*_1$.
The containment of $B$ in $C$ yields, via Hahn-Banach a continuous map, $f: C^*_1 \to B^*_1$. Taking any lifting,
$s: X \to C^*_1$ yields a contractive, linear map, $T_s: C \to C(X)$ by setting $T_s(v)(x) = s(x)(v)$
and it is readily checked that this map extends $T.$ Thus, $C(X)$ is injective in the category of Banach spaces and
contractive linear maps.

{\bf $\left(  5\right)  \Rightarrow\left(  1\right).$} Let $S_1 \subseteq S_2$ be unital operator systems and
let $\phi: S_1 \to C(X)$ be a unital, completely positive map. Then $\phi$ is contractive and so there exists a
contractive, linear extension $\psi: S_2 \to C(X).$ But since $\psi$ is an extension, it is unital and unital contractive
maps are positive. Finally, since $C(X)$ is abelian every positive map is completely positive and so $C(X)$ is injective
in the category of operator systems and unital, completely positive maps.

\end{proof}

\begin{remark} \begin{itemize}
\item[(i)] 
The equivalence of (2) and (3) is due to Gleason and was his key motivation for introducing projective spaces.
\item[(ii)]
Nachbin-Goodner-Kelley\cite{Na,Goo,Ke} prove that a Banach space is
injective in the category of Banach spaces and contractive, linear
maps if and only if it is isometrically isomorphic to $C(X)$ where $X$
is extremally disconnected. However, it is easy to prove that if a
Banach space is injective, then it is linearly isometrically isomorphic to $C(X)$ for some
compact, Hausdorff space $X$. To see this assume that $Y$ is injective and
consider the canonical embedding of $Y \subseteq C(Y^*_1)$, into the
continuous functions on the unit ball of the dual space. Since $Y$ is
injective, there is a contractive, projection $\psi: C(Y^*_1) \to Y.$
Using this projection we may define a binary relation on $Y$ by setting, $y_1
\circ y_2 = \psi(y_1 \cdot y_2),$ where $\cdot$ is the product in
$C(Y^*_1)$. One checks that this binary relation is a product that makes $Y$ into an abelian
C*-algebra and hence, $Y$ is isometrically isomorphic to $C(X)$ for
some compact, Hausdorff space $X$. Thus, the equivalence of (3) and
(5), together with this observation and Gleason's theorem yields the
Nachbin-Goodner-Kelley theorem.
\item[(iii)]
An abelian C*-algebra is an AW*-algebra if and only if every bounded set of self-adjoint elements has a least upper bound. It is easily seen that given the situation of $(3)$, the least upper bound of the set will do for $a.$ Thus, every abelian AW*-algebra satisfies $(4)$. Conversely, given any bounded set, $\cl E$, if we let $\cl F$ denote the set of all upper bounds, then the element $a$ that is assumed to exist in $(4)$ is easily seen to be a least upper bound. Thus, $(3)$ is equivalent to requiring that $\cl A$ be an AW*-algebra.
\item[(iv)]
Isbell\cite{Is} proves that $C(X)$ is injective in the category of Banach spaces and contractive linear maps if and only if
$C(X)$ is an abelian $AW^*$-algebra. Thus, Isbell's theorem follows from the equivalence of $(4)$ and $(5)$.
\item[(v)]
Finally, Stone's theorem says that the maximal ideal space of a complete Boolean algebra is an extremally disconnected space. If $\cl B$ is a complete Boolean algebra, then it can be shown that $C^*(\cl B)$ is an abelian, AW*-algebra, and, hence, Stone's theorem follows from Gleason's theorem and the equivalence of (3) and (4).
\item[(vi)] 
Hamana\cite{H1} uses Isbell's result to argue that $C(X)$ is an injective operator system if and only if
 $C(X)$ is an abelian $AW^*$-algebra. This follows from the equivalence of (2) and (4).
\end{itemize}
 \end{remark}
 
 We now illustrate how some of the results of \cite{DP1,DP2,Pea}, concerning continuous matrix-valued functions on Stonian spaces, can be deduced more readily, using Gleason's characterization of these spaces as projective. Our proofs mimic the proofs given by Azoff\cite{Az} in the context of von Neumann's principle of measurable selection. In fact, Azoff seems to have recognized that the principle of measurable selection is really a statement about projectivity in a category with Borel measurable maps.


\begin{prop} \cite[Theorem 1]{DP1} Let $X$ be a Stonian space and let $p_0(x),..., p_{n-1}(x)$ be continuous complex-valued functions on $X$, then there exist continuous functions, $\lambda_1(x),..., \lambda_{n-1}(x)$ on $X$, such that $$z^n + \sum_{i=0}^{n-1} p_i(x)z^i = \prod_{i=1}^n (z- \lambda_i(x)).$$
That is, every monic polynomial, whose coefficients are continuous functions on a Stonian space, can be factored with roots that are continuous functions.
\end{prop}
\begin{proof}
Let $\cl P_n$ denote the space of monic polynomials of degree $n$, identified with $\bb C^n,$ and consider the map, $h: \bb C^n \to \cl P_n$ defined by $$h((\lambda_1,...,\lambda_n))= \prod_{i=1}^n (z- \lambda_i).$$
Note that $h$ is onto, and that if $C$ is a compact subset of $\cl P_n$, then $h^{-1}(C) \subseteq \bb C^n$ is compact.

Setting, $p(x) = z^n + \sum_{i=0}^{n-1}p_i(x)z^i$ defines a continuous map, $p:X \to \cl P_n.$ Since $X$ is projective, this map lifts to $\Lambda: X \to \bb C^n,$ with $\Lambda(x)= (\lambda_1(x),...,\lambda_n(x))$ and these components define the root functions.
\end{proof}

\begin{prop} \cite[Corollary 3.3]{DP1} Let $X$ be Stonian and let $h_{i,j} \in C(X), i,j=1,...,n$ be such that the matrix-valued function, $H(x)= (h_{i,j}(x))$ is Hermitian. Then there are continuous matrix-valued functions, $U(x)$ and $D(x)$, with $U$ unitary-valued and $D$ diagonal-valued, such that $$H(x)= U(x)^*D(x)U(x).$$
That is, every Hermitian-valued function on a Stonian space can be continuously diagonalized.
\end{prop}
\begin{proof} Let $\cl U_n \subseteq M_n$ denote the unitary matrices,  let $\cl D_n \subseteq M_n$ denote the diagonal matrices and let $\cl H_n \subseteq M_n$ denote the Hermitian matrices and consider the continuous onto map, $q: \cl U_n \times \cl D_n \to \cl H_n$ defined by $q((U,D)) = U^*DU.$ Again, it can be shown that if C is a compact subset of $\cl H_n,$ then $q^{-1}(C)$ is compact.
Thus, since $X$ is projective, the continuous function, $H:X \to \cl H_n$ has a continuous lifting, $(U(x), D(x)).$
\end{proof}

In a similar fashion, one can use the projectivity of Stonian spaces
to prove \cite[Theorem 3]{Pea} and
\cite[Theorem 1]{DP2}.

We now present the simplest example of a projective space. Given a locally, compact Hausdorff space
$E$ we let $\beta E$ denote the Stone-Cech compactification of $E$.

\begin{prop} Let $E$ be an arbitrary set, endowed with the discrete
  topology. Then $\beta E$ is projective and consequently, extremally disconnected.
\end{prop}
\begin{proof}
Let $P= \beta E, X, Y, h, f$ be as in the above definition. For each $e \in E$ choose $r(e) \in Y$
such that $h \circ r(e) = f(e).$ By the universal property of the Stone-Cech compactifaction, the map
$r$ extends uniquely to a continuous function from $\beta E$ to $Y$, that we still denote by $r$.
Since, $E$ is dense in $\beta E$, by continuity we will have that $h \circ r(p) = f(p)$ for every
$p \in \beta E.$
\end{proof}

The following characterization of projective spaces is often useful.

\begin{prop} Let $P$ be a compact, Hausdorff space. Then $P$ is projective if and only if for every
compact, Hausdorff space $W$ and continuous $g: W \to P,$ onto, there exists a continuous $s: P \to W$
such that $g \circ s(p) =p.$
\end{prop}
\begin{proof} Assume that $P$ is projective and let $s$ be a lifting of the identity map on $P.$

Conversely, assume that $P$ has the above property and let $h, f, X$ and $Y$ be as in the definition of
projectivity. Let $W= \{ (p,y) \in P \times Y: f(p) = h(y) \}$ and define $g: W \to P$ by $g(p,y)=p$
and $q: W \to Y$ by $q(p,y)=y.$ If $s: P \to W$ is as above then $r= q \circ s$ is a lifting of $f.$ 
\end{proof}
 
\begin{defn} Let $X$ be a compact, Hausdorff space, we call a pair
  $(C,f)$ a {\bf cover} of $X,$ provided that $C$ is a
  compact, Hausdorff space and $f: C \to X$ is a continuous map that is onto
  $X.$ We call $(C,f)$ an {\bf essential cover} of $X,$ if
  it is a cover and whenever $Y$ is a compact, Hausdorff space, $h:Y
  \to C$ is continuous and $f(h(Y))=X,$ then necessarily $h(Y)=C.$
We call $(C,f)$ a {\bf rigid cover} of $X$, if it is a cover and the
  only continuous map, $h:C \to C$ satisfying $f(h(c))= f(c)$ for
  every $c \in C$ is the
  identity map.
\end{defn}

\begin{prop} Let $X$ be a compact, Hausdorff space and let $(C,f)$ be
  an essential cover of $X$. Then $(C,f)$ is a rigid cover of $X.$
\end{prop}
\begin{proof}
Let $h: C \to C$ satisfy $f(h(c)) = f(c)$ for every $c \in C.$ Let
$C_1=h(C)$ which is a compact subset of $C$
that still maps onto $X$. The inclusion map of $i:C_1 \to C$
satisfies, $f(i(C_1))=X$ and hence must be onto $C$. Thus, $h(C)=C.$ 

Next, we claim that if
$U \subseteq C$ is any non-empty open set, then $U \cap h^{-1}(U)$ is non-empty. For assume to the contrary,
and let $F = C\backslash U.$ Then $F$ is compact and given any $c \in U$ there exists $y \in h^{-1}(U)$ with $h(y)=c$.
Hence, $y \in F$ and $f(c) = f(h(y))= f(y)$. Thus, $f(F)=X,$ again contradicting the essentiality
of $C$. Thus, for every open set $U$, we have that $U \cap h^{-1}(U)$
is non-empty. 

Now fix any
$c \in C$ and for every neighborhood $U$ of $c$ pick $x_U \in U \cap h^{-1}(U).$ We have that the
net $\{ x_U \}$ converges to $c.$ Hence, by continuity, $\{ h(x_U) \}$ converges to $h(c).$ But since
$h(x_U) \in U$ for every $U,$ we also have that $\{ h(x_U) \}$
converges to $c.$ Thus, $h(c) =c$ and since $c$ was arbitrary,
$C$ is rigid.
\end{proof}

We will see that the converse holds when $C$ is also projective.

The following result shows that projective covers always
exist. Gleason's proof of this result \cite{Gl} used that every unital, abelian
$C^*$-algebra, $C(X)$ could be embedded into a complete Boolean
algebra and then appealed to a theorem of Stone's \cite{St} that
the maximal ideal space of a complete Boolean algebra is extremally disconnected.

\begin{prop} \label{2.12} Let $X$ be a compact, Hausdorff space, then there exists a projective space $P$ and a
continuous $f: P \to X,$ onto.
\end{prop}
\begin{proof} Let $E=X,$ endowed with the discrete topology and let $P= \beta E.$
\end{proof}

The following results relate the concepts of essential cover and rigid cover and give us a means of
obtaining Gleason's projective envelope that stays entirely in the context of topology, hence
avoiding the deep results from Boolean algebras that Gleason used.

\begin{prop} \label{2.13} Let $(C,f)$ be a cover of $X$ with $C$ a projective space. Then $(C,f)$ is an essential  cover if and only if $(C,f)$ is a rigid cover.
\end{prop}
\begin{proof}  We already have that an essential cover is always a rigid cover. So assume that $(C,f)$ is a rigid cover. Let $h: Y \to C$ with $f(h(Y))=X.$ Since $C$ is projective, there exists a map $s: C \to Y$
with $(f \circ h) \circ s = f.$ We have $h \circ s: C \to C$ and $f(h \circ s(c)) = f(c)$ and so by
rigidity, $h \circ s(c) =c$ for every $c \in C.$ In particular, $h$ must be onto and so $C$ is essential.
\end{proof}

\begin{remark} In the linear theory, essential and injective is,
  generally, equivalent to rigid and injective \cite{Co}.
The above result shows the analogue in the topological setting. 
However, we will later see that when one includes a G-action, 
then G-projective and G-essential do not imply G-rigid.
\end{remark}

Note that if $(C,f)$ is a cover of $X,$ then the compact subsets of C that still map onto X satisfy
the hypotheses of Zorn's lemma, so that one can always choose a minimal such subset.

\begin{prop} \label{2.15} Let $(Y,f)$ be a cover of $X$ and let $C \subset Y$ be a minimal, compact subset of $Y$ that 
maps onto $X.$ Then $(C,f)$ is a rigid, essential cover of $X.$
\end{prop}
\begin{proof} First, we prove essential. Given any compact, Hausdorff space $Z$ and $h: Z \to C$ such that
$f(h(Z))=X$, we have that $h(Z) \subseteq C$ is compact and hence $h(Z)=C$ by minimality.

Since, $(C,f)$ is an essential cover of $X$, by the above results it
is also a rigid cover.
\end{proof}

The following result is only slightly different from results of Gleason and Gonshor. However, the proof below
eliminates the need for the use of Stone's representation theorem for complete Boolean algebras and proves that the space we seek is really a minimal projective space that is a cover.

\begin{thm}\label{existhm} Let $X$ be a compact, Hausdorff space, let $(P,f)$ be any
  projective space that is a cover of $X$ and let $C \subseteq P$ be a minimal,
  compact subset among all sets satisfying $f(C)=X.$ Then $(C,f)$ is a
  rigid, essential projective cover of $X$. Moreover, if $(C', f')$ is a cover of $X$ that is projective and either rigid or essential, then there exists  a homeomorphism
$h: C \to C'$ such that $f'(h(c)) = f(c)$ for every $c \in C.$
\end{thm}
\begin{proof} By Proposition \ref{2.12} we know that a cover $(P,f)$ of $X$ exists with $P$ projective. If we let
$C \subseteq P$ be any minimal compact subset of $P$ that still maps onto $X$, then $(C,f)$ will be
a rigid, essential cover by the above. Since, $P$ is projective there exists $h: P \to C$ such
that $f(h(p)) = f(p)$ for every $p \in P.$ Since $C$ is essential, $h(P)=C$ and since $C$ is rigid,
$h(c)=c$ for every $c \in C.$ Thus, $h$ is an idempotent map of $P$ onto $C$, that is, $C$ is a retract
of $P$. From this it follows easily that $C$ is projective.

Finally, if $(C',f')$ is a projective cover of $X$, then using the projectivity of $C$ and $C'$ one obtains
maps $h:C \to C'$ and $h': C' \to C$ such that $f' \circ h =f$ and $f \circ h' = f'.$ From rigidity
the rigidity of $C$,
it follows that $h' \circ h$ is the identity on $C$ and hence $h$ is one-to-one. If $C'$ is rigid, then
$h \circ h'$ is the identity on $C'$ and so $h$ and $h'$ are mutual inverses and homeomorphisms.
If $C'$ is essential, then $h(C)=C'$ and so $h$ is one-to-one and hence a homeomorphism.
\end{proof}

\begin{defn} Let $X$ be a topological space, we call $(C,f)$ a {\bf
    projective cover of X,} provided that $(C,f)$ is an essential cover of $X$ and $C$ is projective.
\end{defn}

Thus, by the above results, every compact, Hausdorff space $X$ has a
projective cover $(C,f)$ and $(C,f)$ is also a rigid cover. Moreover, the projective
cover is unique, up to homeomorphisms that commute with the covering maps and we have a prescription for how to obtain it.

Since the category of operator systems strictly contains the category of abelian $C^*$-algebras
and contains more injective objects, it is conceivable that the injective envelope of a $C(X)$ space
in this larger category could be smaller. Hamana\cite{H2} remarks that the two are the same.  Since we will
be using this fact, we prove it below. Again, the use of projectivity leads to a simpler proof.

\begin{cor} Let $X$ be a compact, Hausdorff space and let $(K,f)$ be its  projective
cover. Then $C(K)$ is *-isomorphic to $I(C(X)),$ Hamana's injective envelope of $C(X)$ in the category of operator systems
and unital, completely positive maps, via a *-isomorphism that fixes $C(X).$
\end{cor}
\begin{proof} Since $C(K)$ is injective in the category of operator systems and composition with $f$
induces a *-monomorphism $\pi_f : C(X) \to C(K),$ there will exist a complete order isometry of $I(C(X))$
into $C(K)$ together with a completely positive projection of $C(K)$ onto this subspace. This projection
endows the image of $I(C(X))$ with an abelian product, making $I(C(X))$ into an abelian $C^*$-algebra.
Thus, $I(C(X))= C(Z)$ for some projective space $Z$ and  the inclusion of $C(X)$ into $C(Z)$ is given
by an onto continuous function $h: Z \to X.$ Projectivity of $Z$ and $K$ yields continuous functions between
$K$ and $Z$ and rigidity forces these maps to be one-to-one and onto.
\end{proof}

One consequence of the preceding result is that if $\mathcal{A}$ is a
commutative unital C*-algebra contained in a commutative unital injective
C*-algebra $\mathcal{B}$, then there is a copy of the injective hull of
$\mathcal{A}$ that contains $\mathcal{A}$ and sits inside $\mathcal{B}$ as a
C*-subalgebra. However, the commutativity of $\mathcal{B}$ is not necessary here.

\begin{prop}
Suppose $\mathcal{S}\subset B\left(  H\right)  $ is a unital injective
operator system, $\mathcal{C}$ is a C*-subalgebra of $\mathcal{S}$ and
$\mathcal{C}^{\prime}$ is injective. Then $\mathcal{S}\cap\mathcal{C}^{\prime
}$ is injective.
\end{prop}

\begin{proof}
Let $\varphi:B\left(  H\right)  \rightarrow\mathcal{S}$, and $\psi:B\left(
H\right)  \rightarrow\mathcal{C}^{\prime}$ be completely positive projections.
Since $\varphi$ is a $\mathcal{C}$-bimodule map, $\varphi\left(
\mathcal{C}^{\prime}\right)  \subset\mathcal{C}^{\prime}.$ Thus $\varphi
\circ\psi:B\left(  H\right)  \rightarrow\mathcal{S}\cap\mathcal{C}^{\prime}.$
However, both $\varphi$ and $\psi$ fix $\mathcal{S}\cap\mathcal{C}^{\prime},$
hence $\varphi\circ\psi$ is a completely positive projection onto
$\mathcal{S}\cap\mathcal{C}^{\prime}.$ Thus $\mathcal{S}\cap\mathcal{C}%
^{\prime}$ is injective.
\end{proof}

\begin{cor}
Suppose $\mathcal{B}$ is a unital injective C*-algebra. Then every maximal
abelian C*-subalgebra is injective.
\end{cor}
\begin{proof} Regard $\cl B$ as a $C^*$-subalgebra of $B(\cl H)$ and let $\cl C \subseteq \cl B$
be a maximal abelian $C^*$-subalgebra of $\cl B$. Extend $\cl C$ to a maximal abelian
$C^*$-subalgebra $\cl M$ of $B(\cl H).$ Then $\cl M^{\prime} = \cl M$ is injective and hence
$\cl B \cap \cl M = \cl C$ is injective.
\end{proof}

\begin{thm}
Suppose $\mathcal{A}$ is a commutative C*-algebra, $\mathcal{B}$ is an
injective C*-algebra and $\pi:\mathcal{A}\rightarrow\mathcal{B}$ is a 1-1
$\ast$-homomorphism. Then $\pi$ extends to a 1-1 $\ast$-homomorphism
$\rho:I\left(  \mathcal{A}\right)  \rightarrow\mathcal{B}.$
\end{thm}

\begin{proof}
Without loss of generality we can assume $\mathcal{A}\subset\mathcal{B}$ and
$\pi$ is the inclusion map. Let $\mathcal{C}$ be a maximal abelian selfadjoint
subalgebra of $\mathcal{B}$ that contains $\mathcal{A}.$ It follows from the
preceding corollary that $\mathcal{C}$ is injective in the category of abelian $C^*$-algebras
and *-homomorphisms. 
Thus, $\pi$ extends to a $\ast$-homomorphism $\rho:I(\mathcal{A})
\rightarrow\mathcal{C}.$ Similarly, injectivity of $I(\cl A)$ in the category gives a *-homomorphism
of $\cl C$ to $I(\cl A)$ that fixes $\cl A.$ By rigidity this composition must be the identity on
$I(\cl A)$ and so $\rho$ was one-to-one.
\end{proof}

The above result has interesting consequences for the interplay between measure and topology. Recall that a set is called {\bf meager} or {\bf first category} if it is a countable union of nowhere dense sets. 
Given a compact, Hausdorff space $X$, let $B(X)$ denote the $C^*$-algebra of bounded, Borel functions on $X$ and let $\cl M(X) \subseteq \cl B(X)$ denote the ideal of Borel functions that vanish off a meager set. 
The quotient, $\cl D(X) = \cl B(X)/\cl M(X)$ is called the {\em
  Dixmier algebra,} in honor of J. Dixmier who proved \cite{Di} that
$\cl D(X)$ is the injective envelope of $C(X).$ More precisely, $\cl
D(X)$ is an injective $C^*$-algebra, the inclusion of $C(X) \subseteq
\cl B(X)$ restricts to a *-monomorphism on $C(X)$ when one passes to
the quotient, so that we may regard $C(X) \subseteq \cl D(X)$ and
there is a *-isomorphism between $I(C(X))$ and $\cl D(X)$ that fixes
$C(X)$. Thus, identifying $\cl D(X) \equiv C(K)$ and the incusion map
as the *-homomorphism given by composition $\Pi_f:C(X) \to C(K)$ for some continuous function, $f:K \to X,$ we see that $(K,f)$ is the projective cover of $X$.

Recall also that there exist meager subsets of [0,1] that are nonmeasurable and meager subsets that are Borel sets of arbitrary measure("fat" Cantor sets will do). Moreover, there exist Borel sets of measure 0, i.e., null sets, that are not meager. For these reasons there are no "natural" maps between the algebras $\cl D([0,1])$ and $L^{\infty}([0,1])= \cl B([0,1])/\cl N([0,1]),$ where $\cl N([0,1])$ denotes the ideal of Borel functions that vanish off a null set.
Nevertheless, we have the following:

\begin{thm} There exist *-homomorphisms, $\pi: \cl D([0,1]) \to L^{\infty}([0,1])$ and $\rho: L^{\infty}([0,1]) \to \cl D([0,1])$ such that $\pi(f+ \cl M([0,1]))= f + \cl N([0,1]), \rho(f+ \cl N([0,1])) = f + \cl M([0,1])$ for every  $f \in C([0,1])$ and $\rho \circ \pi(a) =a$ for every $a \in \cl D([0,1]).$
\end{thm}
\begin{proof} The existence of the *-homomorphisms fixing $C([0,1])$ follows from the injectivity of the two algebras in the category of abelian C*-algebras and *-homomorphisms, i.e., from the projectivity of their maximal ideal spaces.
The fact that the composition must be the identity on $\cl D([0,1])$ follows from Dixmier's result that $\cl D([0,1])$ is the injective envelope of $C([0,1])$ and the rigidity of the injective envelope.
\end{proof}


\section{G-Projective Spaces}

In this section we attempt to generalize the results of the previous section to a dynamical
situation. We assume throughout this section that $G$ is a discrete group with identity $e$.
By an {\bf action of G on a topological space X,} we mean a homomorphism of G into the group of homeomorphisms
of X that sends the identity of G to the identity map. 
Given $g \in G$ and $x \in X$, we denote the image of $x$ under the homeomorphism
corresponding to $g$ by $g \cdot x.$ We shall call a compact, Hausdorff space X, equipped with an
action by $G$ a {\bf G-space.} Given two G-spaces, X and Y by a {\bf G-map} we mean a continuous map
$f: X \to Y$ such that $f( g \cdot x) = g \cdot f(x).$ Such a map is
also called {\bf G-equivariant.}

We define {\bf G-cover, G-projective, G-rigid cover,} and {\bf G-essential cover} by analogy
with the earlier definitions, by simply replacing "compact, Hausdorff spaces" by "G-spaces" and "continuous function" by "G-map".

Let $X$ be any G-space and let $(P,r)$ be its
 projective cover. Using the rigidity of $P$, for every $g \in G$ the homeomorphism $x \to g
  \cdot x$ extends uniquely to a homeomorphism of $P$, which we still
  denote by $p \to g \cdot p$ satisfying $r(g \cdot p) = g \cdot r(p)$
  and this collection of homeomorphisms makes $P$ into a G-space and
  $r$ a G-equivariant map. Thus, the projective cover of any G-space
  is again a G-space and the covering map is a G-map.  That is, the projective cover of $X$ is a G-cover. 
Moreover, we have that $(P,r)$ is a G-essential cover of X.
But, unfortunately, $P$ is generally not G-projective.

In fact, if we consider the simplest example, namely, a one point
 space, $X= \{ x \},$ with the trivial G-action. Then $X$ is clearly
 projective as a topological space, but it is not G-projective. This
 can be seen by considering the identity map from X to X and the map
 from some G-space  Y onto X, where Y has no fixed points.
This map has no G-equivariant lifting.

Thus, in general, although the projective cover of a space is a G-essential cover, it need not be a G-projective space.
In fact, this example shows that we will need to amend our 
definition/expectations of G-projective covers.

\begin{prop} Let G be a non-trivial, countable discrete group and let $X= \{ x \}$ be a singleton with the trivial G-action. Then $X$ does not have a G-projective, G-rigid cover.
\end{prop}
\begin{proof} Assume that $P$ is a G-projective, G-rigid cover of $X$ and let $f:P \to X.$
Fix $g_1 \in G, g_1 \ne e$ so that $h:P \to P, h(p) = g_1 \cdot p$ is a G-map.
Since, $f \circ h =f,$ by G-rigidity, $h$ must be the identity. Thus, the G-action on P must be trivial.

Now let $Y$ be any G-space and $q:Y \to X.$ Since $P$ is G-projective, there exists a G-map, $r:P \to Y.$ But since the action of G is trivial on $P$, the points in $r(P)$ must be fixed by the G-action on $Y$. 

Thus, $Y$ has a fixed point and so every space that G acts on must have a fixed point. Such a group is called {\bf extremely amenable} and by \cite{GP}, no countable group is extremely amenable. 

Hence, $X$ has no G-projective, G-rigid cover.
\end{proof}

\begin{defn} Let G be a countable, discrete group and let X be a G-space, we call (P,f) a {\bf G-projective cover of X,} provided that P is G-projective, $f:P \to X$ is a G-map and (P,f) is a G-essential cover of X.
\end{defn}

We have been unable to prove that every G-space has a G-projective cover, but we will show that certain "minimal" G-spaces have G-projective covers and derive various properties of G-projectivity that are related to topics in topological dynamics.

It is easy to see (and we prove this below) that
G-projective spaces exist and that every G-space has a G-cover that is a G-projective space.
To this end, let W be any (discrete) set. We define an action of G on 
$G \times W$ by setting $g_1 \cdot (g,w)= (g_1g,w),$ for any $g_1 \in G.$ It is clear that each element of
$G$ defines a permutation of the elements of $G \times W,$ and hence by the universal properties of
the Stone-Cech compactification, this permutation extends to a homeomorphism of $\beta (G \times W)$
and this collection of homeomorphisms defines an action of G on $\beta (G \times W),$ making it into
a G-space.

\begin{prop} \label{2.3} Let W be any set. Then $\beta (G \times W)$ is a projective G-space.
\end{prop}
\begin{proof} Let $X,Y$ be G-spaces, let $f: Y \to X$ be an onto G-map and let $h: \beta(G \times W) \to X$
be a G-map. Let $e \in G$ denote the identity element and choose elements $y_w \in Y$ such that
$h((e,w))= f(y_w).$ Define a map $r: G \times W \to Y$ by setting $r((g,w)) = g \cdot y_w.$ 
By the universal properties, this extends to a unique continuous function from $\beta (G \times W)$ to
$Y$ that is easily seen to be a G-map.
\end{proof}

\begin{prop} \label{2.4} Every G-space has a G-cover that is a G-projective space.
\end{prop}
\begin{proof} Let $X$ be a G-space, let $X_d$ denote $X$ with the discrete topology. The map $(g,x) \to g \cdot x$
extends to a G-map from $\beta (G \times X_d)$ to $X$.
\end{proof}

We now consider some analogues of our earlier results.

\begin{prop} \label{2.5}Let $X$ be a G-space. Then $X$ is G-projective if and only if $C(X)$ is G-injective, i.e., is injective in the
category whose objects are abelian $C^*$-algebras, equipped with G-actions and whose morphisms are
G-equivariant *-homomorphisms.
Every unital abelian C*-algebra equipped with a G-action can be embedded *-monomorphically and G-equivariantly into a unital, abelian C*-algebra that is G-injective.
\end{prop}

\begin{prop} \label{2.6} Let $X$ be a G-projective space. Then $C(X)$ is
  G-injective in each of the following categories:
\begin{enumerate}
\item the category of G-operator spaces and completely contractive
  G-equivariant maps,
\item the category of G-operator systems and completely positive
  G-equivariant maps,
\item the category of G-Banach spaces and contractive G-equivariant linear
  maps.
\end{enumerate}
\end{prop}

\begin{proof}  The proofs of all three of these statements is similar
  to the proof that (3) implies (5) in Theorem 2.4.
\end{proof}

\begin{cor} Let W be any set and let G be any discrete group, then
  $\ell^{\infty}(G \times W) =
  C(\beta (G \times W))$ is G-injective in each of the three categories.
\end{cor}

\begin{prop} Let G be a countable, discrete group, then $\beta (G)$ has
no fixed points.
\end{prop}
\begin{proof} Assume that $\beta G$ has a fixed point, $\omega_0$. Let $Y$ be any G-space and let $y_0 \in Y.$
We can define a G-equivariant continuous map, $f: \beta G \to Y$ by setting $f(g) = g \cdot y_0$ and extending.
Now $f(\omega_0)$ is a fixed point of $Y$, since $g \cdot f(\omega_0) = f(g \cdot \omega_0) = f( \omega_0).$ Thus, every space that G acts on has a fixed point and we are again done by \cite{GP}.
\end{proof}

\begin{defn} A G-space is {\bf extremally G-disconnected} if the closure of every G-invariant open set
is open.
\end{defn}

\begin{prop} If $X$ is G-projective, then $X$ is extremally G-disconnected.
\end{prop}
\begin{proof} Given any open, G-invariant subset $U$ of $X$, let $Y$ be the disjoint union of
the G-spaces $X\backslash U$ and $\Bar(U).$ Define a G-map from $Y$ onto $X$ by the two
inclusions. Since $X$ is G-projective there is a G-map from $h:X \to Y.$ Clearly, for $u \in U$,
we have $h(u) = u.$  Hence, $h^{-1}(\Bar(U) ) = \Bar(U),$ but since $\Bar(U)$ is open
in $Y$ it must be open in $X$.
\end{proof}

\begin{remark} The space consisting of a single point together with the trivial action is extremally
G-disconnected, but is not G-projective.
\end{remark}

\begin{prob} Is there some property of the G-action such that together with extremally G-disconnected
yields G-projective? 
\end{prob}

\begin{prop} Let $X$ be a G-space, let $C$ be a G-projective space such that $(C,f)$ is a G-rigid cover of $X$, then $(C,f)$ is a
G-essential cover.
\end{prop}
\begin{proof} The proof proceeds exactly as in Proposition \ref{2.13}.
\end{proof}

\begin{prop} Let $(Y,f)$ be a G-cover of a G-space $X$, and let $C \subset Y$ be a minimal, compact,
G-invariant subset of $Y$ that maps onto $X.$ Then $(C,f)$ is a G-essential cover of $X$.
\end{prop}
\begin{proof} The proof proceeds exactly as in Proposition \ref{2.15}.
\end{proof}

\begin{prob} If Y is G-projective, then is C G-projective? 
\end{prob}

If true, then this would prove that every G-space has a G-projective, G-essential cover. In our proof of the existence of projective covers(Theorem \ref{existhm}), we used  that the cover was rigid in a key way to get that C was projective.
Gleason's proof \cite{Gl} uses Stone's Boolean algebra representation theorem. Thus, for that approach, we would need to first
generalize Stone's theorem to a G-equivariant Boolean algebras. 



We now proceed to the one key case for which we can prove the existence of a G-projective.

\begin{defn} We say that a G-space X is {\bf minimal} if the G-orbit of every point is dense in X.
\end{defn}

It is easy to see that $X$ is a minimal G-space if and only if $C(X)$ is {\bf G-simple}, i.e., has no non-trivial two-sided G-invariant ideals.

To obtain the G-projective cover in this case, we will need to use some results on algebra in $\beta G,$ taken from \cite{HS}. As we saw earlier, left multiplication on G, extends to define a G-action on $\beta G.$ By the universal properties of the Stone-Cech compactification, for each $\omega \in \beta G,$ the map, $g \to g \cdot \omega,$ extends to  continuous map, $\rho_{\omega}: \beta G \to \beta G.$ Setting $\omega_2 \cdot \omega_1 = \rho_{\omega_1}(\omega_2),$ defines an associative product on $\beta G$ that extends the product on G and makes $\beta G$ a right topological semigroup, i.e., one for which multiplication on the left is continuous for each fixed element on the right, with the same identity as G \cite[Chapter 4]{HS}. Moreover, the corona $G^*= \beta G \setminus G$ is a two-sided ideal in this semigroup \cite[Corollary 4.33]{HS}. Thus, $\beta G$ has minimal left ideals that are not all of $\beta G.$ By \cite[Corollary 2.6]{HS} and \cite[Theorem 2.9]{HS} every minimal left ideal is closed and of the form $(\beta G) \cdot \omega$ with $\omega$ and idempotent element. Note that necessarily, $\omega \in G^*.$
Finally, by \cite[Theorem 2.11c]{HS} any two minimal left ideals, $L$ and $L^{\prime}$ are homeomorphic, and if $\omega \in L$ is any element, then $\rho_{\omega}: L^{\prime} \to L$ is a homeomorphism. Note that $\rho_{\omega}$ is also a G-map.  

\begin{thm} Let $G$ be a countable discrete group, and let $L$ be a minimal left ideal in $\beta G,$ then $L$ is G-projective.
\end{thm}
\begin{proof} Choose an idempotent $\omega_1$ such that
$L = (\beta G) \cdot \omega_1,$ and look at the G-map, $\rho_{\omega_1}: \beta G \to L.$ By associativity of the product and the fact that $\omega_1$ is idempotent, we have that $\rho_{\omega_1} \circ \rho_{\omega_1} = \rho_{\omega_1}.$
Also, for $z= \omega \cdot \omega_1 \in L,$ we have that $\rho_{\omega_1}(z) = z.$

This shows that $\rho_{\omega_1}$ is a G-equivariant retraction of $\beta G$ onto $L.$ Now a little diagram chase, shows that $\beta G$ G-projective implies that $L$ is G-projective.
\end{proof}

Note that if $X$ is any G-space, then the action of G on X extends to an action of $\beta G$ on $X$. To see this, note that for each fixed $x \in X,$ the map $g \to g \cdot x,$ extends to a continuous function, $f_x: \beta G \to X$ and we set $\omega \cdot x= f_x(\omega).$ Since $f_x(g_1 \cdot g_2) = f_{g_2 \cdot x}(g_1),$ it follows by taking limits along nets first, that $(g_1 \cdot \omega_2) \cdot x = g_1 \cdot ( \omega_2 \cdot x),$ and then that $(\omega_1 \cdot \omega_2) \cdot x = \omega_1 \cdot (\omega_2 \cdot x),$ i.e., that the action is associative.

\begin{thm} \label{min} Let G be a countable, discrete group, let X be a minimal G-space, and let L be a minimal left ideal in $\beta G.$
If we fix any point, $x_0 \in X$ and let $f_0: \beta G \to X,$ be the map $f_0(\omega) = \omega \cdot x_0,$ then $(L,f_0)$ is a G-projective cover of X.
\end{thm}
\begin{proof} We have that $f_0$ is a G-map by the associativity of the action. Write $L= (\beta G) \cdot \omega_0,$ with $\omega_0$ idempotent, then $f_0(L) = \{ (\omega \cdot \omega_0) \cdot x_0 : \omega \in \beta G \} = \{ \omega \cdot ( \omega_0 \cdot x_0): \omega \in \beta G \} = X,$ since the orbit of $\omega \cdot x_0$ is dense and the image of $L$ is compact. Hence, $(L, f_0)$ is G-projective and a G-cover of $X$.  

It remains to show that it is a G-essential cover. So assume that $Y$ is a G-space, that $h:Y \to L$ is a G-map and that $f_0 \circ h(Y)= X.$ Note that since $h$ is a G-map, $G \cdot h(Y) \subseteq h(Y).$ Hence, $h(Y)$ is a closed left ideal and so by minimality, $h(Y) =L.$  Thus, $L$ is a G-essential cover.
\end{proof}

Applying the duality between G-spaces and unital, abelian C*-algebras quipeed with G-actions leads to the following:

\begin{cor} Let G be a countable discrete group, and let $\cl A$ be a unital abelian C*-algebra equipped with a G-action. If $\cl A$ is G-simple, then there is an G-equivariant *-monomorphic embedding of $\cl A$ into a G-injective unital abelian C*-algebra $\cl B$ with the property that whenever $\cl C$ is another C*-algebra equipped with a G-action and $\pi: \cl B \to \cl C$ is a G-equivariant *-homomorphism that is one-to-one on $\cl A,$ it follows that $\pi$ is one-to-one.
\end{cor}
\begin{proof} This is a combination of the fact that minimal is the same as G-simple, together with a translation of G-essential from the topological category to the category of abelian C*-algebras.
\end{proof} 

In Ellis \cite{E} and also in \cite{HS} it is proven that the minimal left ideals in $\beta G,$ are the {\bf universal} minimal dynamical systems. It can be shown quite easily, that the G-projectivity of minimal left ideals implies this universal property, but the converse is not so clear. In any case, the connections between Ellis' constructions and projectivity in the topological category seems to have gone unnoticed.


\section{G-injectivity}

We have seen that there is a strong correspondence between projectivity for topological spaces and various notions of injectivity for linear spaces.
In this section, we explore those connections in the presence of an action by a countable discrete group.

The Gelfand duality carries through in this setting to a contravariant functor between the category whose objects are G-spaces with morphisms the G-maps and the category whose objects are abelian C*-algebras equipped with G-actions and whose morphisms are G-equivariant *-homomorphisms. In particular, a space $P$ is G-projective if and only if $C(P)$ is G-injective in this latter category, i.e., G-equivariant *-homomorphisms into $C(P)$ have G-equivariant *-homomorphic extensions.

Moreover, if $P$ is the
projective cover of $X$, then we have seen that we may identify
$I(C(X))= C(P)$ and that $P$ and hence, $I(C(X)),$ is also endowed with a G-action. However, we have also seen that the ordinary projective cover of a space need not be G-projective, and hence not G-injective in the above category.

This leads naturally, to the question of whether or not $C(P)$ is G-injective in some appropriate category and whether or not it is the appropriate notion of the G-injective envelope.

Hamana \cite{H3} \cite{H4} studies injectivity in the presence of a G action for  two larger categories, the category whose objects are operator spaces(respectively, systems)
with a G-action consisting of a group of completely isometric isomorphisms(respectively, unital,
complete order isomorphisms) and morphisms consisting of the completely contractive(respectively, unital
completely positive) G-equivariant maps. In this setting, he proves that every operator space(respectively,
operator system) $V$ has a {\em G-injective envelope $I_G(V),$} that is a ``G-essential'', ``G-rigid'',
"G-injective" extension of $V,$ where the quotation marks are used to
indicate that these definitions have the analogous(but not necessarily
equivalent) meanings in these categories. Moreover, for operator systems, their G-injective envelopes are the same in
either category that Hamana considers. Hamana obtains his injective envelope by first embedding the
G-operator system into a G-operator space that is  injective in the usual sense and G-injective in an appropriate sense. He then obtains
the G-injective envelope as the range of a minimal G-equivariant idempotent map. Because the G-injective envelope is the range of a projection
applied to a space that is injective in the usual sense, it follows that the G-injective envelope is also injective in the usual sense.

For a simple example of Hamana's construction, consider the complex numbers $\bb C$ equipped with
the trivial G-action and $\ell^{\infty}(G)$ equipped with the G-action induced by multiplication on G.
We can consider $\bb C$ as the subspace of $\ell^{\infty}(G)$ consisting of the scalar multiplies of
the identity. An extension of the identity map on $\bb C$ to a (completely) positive G-equivariant map
on $\ell^{\infty}(G)$  would be a G-invariant mean. Hence, $\bb C$ is G-injective if and only if G is 
amenable. On the other hand, Hamana shows that $\ell^{\infty}(G)$ is always G-injective.
Thus, $I_G(\bb C) = \bb C = I(\bb C)$ if and only if $G$ is amenable. For G non-amenable, we have $I(\bb C)= \bb C \subset I_G(\bb C) \subseteq \ell^{\infty}(G).$  It follows readily from his work(we show this below) that $I_G(\bb C) \ne \ell^{\infty}(G),$  but an exact characterization of $I_G(\bb C)$ unanswered.  

\begin{prop} Let G be a non-abelian group, then $\bb C \subset I_G(\bb C) \subset \ell^{\infty}(G).$
\end{prop}
\begin{proof}   We have seen above why $\bb C \ne I_G(\bb C).$ Since $\ell^{\infty}(G)= C(\beta G)$ and $\beta G$ is G-projective, it is G-injective in Hamana's sense by Proposition \ref{2.6}. Thus, the inclusion  $\bb C \subseteq \ell^{\infty}(G)$ extends to a completely order isomorphism of $I_G(\bb C)$ into $\ell^{\infty}(G).$ If we compose this inclusion with the quotient map $q: \ell^{\infty}(G) \to \ell^{\infty}(G)/c_0(G),$ then since the quotient map is a complete isometry on $\bb C$, it must be a complete isometry on $I_G(\bb C).$ This follows from the fact that the G-injective envelope is an essential extension and the definition of essential extension \cite{H3}.

Thus, there exist inclusions of $I_G(\bb C)$ into $\ell^{\infty}(G)$ extending the inclusion of $\bb C$ into $\ell^{\infty}(G)$, but this map can never be onto.
\end{proof}

Because Hamana obtains his G-injective envelope of a G-operator system as a
subspace of a non-commutative crossed product algebra, it is not clear whether or not
$I_G(C(X))$ is even an abelian $C^*$-algebra.

\begin{thm} Let G be an arbitrary discrete group and let $X$ be a
  G-space. Then $I_G(C(X))$ is an abelian C*-algebra.
\end{thm}
\begin{proof} Let $(P,f)$ be a G-projective space that covers  $X$. Then
  composition with $f$ induces a G-equivariant, *-monomorphism,
  $\Pi_f: C(X) \to C(P).$ Since $C(P)$ is G-injective in the category
  of G-operator systems, $\Pi_f$ will extend to a G-equivariant
  complete order isomorphism of $I_G(C(X))$ into $C(P)$. Since
  $I_G(C(X))$ is also G-injective there will exist a G-equivariant
  completely positive projection of $C(P)$ onto $I_G(C(X)).$ Endowing
  $I_G(C(X))$ with the Choi-Effros product induced by this projection
  makes it into a commutative $C^*$-algebra.
\end{proof}

\begin{thm} Let $X$ be a minimal G-space and let $L$ be a minimal left ideal in $\beta G.$ Then there exists a G-equivariant *-monomorphism of $I_G(C(X))$ into $C(L).$
\end{thm}
\begin{proof} Identify $I_G(C(X))=C(Y)$ with $Y$ a G-space, so that the inclusion of $C(X)$ into $C(Y)$ is given as $\Pi_h$ for some onto G-map $h:Y \to X.$
By Theorem \ref{min}, there is an onto G-map, $f:L \to X$ and by the G-projectivity of $L$, we have a G-map $r:L \to Y$ with $h \circ r = f.$
Thus, $\Pi_r: C(Y) \to C(L),$ since $\Pi_r \circ \Pi_h = \Pi_f$ is a *-monomorphism on $C(X)$, by the fact that $\Pi_h$ is an essential extension, $\Pi_r$ must also be a *-monomorphism.
\end{proof}

\begin{prob} Let $X$ be a minimal G-space. If G is non-amenable, is $I_G(C(X))= C(L)$ for $L$ a minimal left ideal in $\beta G$ ? In particular, is $I_G(\bb C) = C(L)$ when G is non-amenable?
\end{prob}

Since we do not have equality for G amenable, we suspect that the equality of $I_G(\bb C)$ and $C(L)$ could be a measure of how badly non-amenable the group is.  

\begin{prob} Let $X$ be a minimal G-space. Give necessary and sufficient conditions to guarantee that $I_G(C(X)) = C(L).$
This would be especially interesting for G amenable.
\end{prob}


\end{document}